\documentclass[11pt]{article}

\usepackage{amsmath,amssymb,amsthm,esint}

\usepackage{color,enumitem,graphicx,cite,mathrsfs,tikz,colortbl}

\usepackage[colorlinks=true,urlcolor=blue,
citecolor=red,linkcolor=blue,linktocpage,pdfpagelabels,
bookmarksnumbered,bookmarksopen]{hyperref}
\usepackage[english]{babel}

\usepackage[pdftex,a4paper,left=2.9cm,right=2.9cm,top=3.2cm,bottom=3.2cm]{geometry}

\usepackage[hyperpageref]{backref}

\newcommand{\bdry}[1]{\partial #1}

\newcommand{\dint}{\ds{\int}}
\newcommand{\dist}[2]{\text{dist}\, (#1,#2)}
\newcommand{\ds}[1]{\displaystyle #1}

\newcommand{\id}[1]{id_{#1}}
\newcommand{\incl}{\subset}

\newcommand{\norm}[2][]{\left\|#2\right\|_{#1}}

\newcommand{\PS}[1]{$(\text{PS})_{#1}$}
\newcommand{\pnorm}[2][]{\if #1'' \left|#2\right|_p \else \left|#2\right|_{#1} \fi}

\newcommand{\restr}[2]{\left.#1\right|_{#2}}
\newcommand{\seq}[1]{\left(#1\right)}
\newcommand{\set}[1]{\left\{#1\right\}}

\newcommand{\vol}[1]{\left|#1\right|}

\newcommand{\N}{\mathbb N}
\newcommand{\R}{\mathbb R}
\newcommand{\RP}{\R \text{P}}
\newcommand{\Z}{\mathbb Z}

\newcommand{\A}{{\cal A}}
\newcommand{\F}{{\cal F}}
\newcommand{\M}{{\cal M}}

\renewcommand{\le}{\leqslant}

\renewcommand{\ge}{\geqslant}

\DeclareMathOperator{\divg}{div}

\newenvironment{enumroman}{\begin{enumerate}
\renewcommand{\theenumi}{$(\roman{enumi})$}
\renewcommand{\labelenumi}{$(\roman{enumi})$}}{\end{enumerate}}

\newenvironment{properties}[1]{\begin{enumerate}
\renewcommand{\theenumi}{$(#1_\arabic{enumi})$}
\renewcommand{\labelenumi}{$(#1_\arabic{enumi})$}}{\end{enumerate}}

\newtheorem{corollary}{Corollary}[section]

\newtheorem{proposition}[corollary]{Proposition}
\newtheorem{theorem}[corollary]{Theorem}

\theoremstyle{remark}
\newtheorem{remark}[corollary]{Remark}

\numberwithin{equation}{section}

\title{\bf Bifurcation and multiplicity results for critical $p$-Laplacian problems}
\author{\bf Kanishka Perera\\
Department of Mathematical Sciences\\
Florida Institute of Technology\\
Melbourne, FL 32901, USA\\
[\bigskipamount]
\bf Marco Squassina\thanks{The second-named author was supported by 2009 MIUR project: ``Variational and Topological Methods in the Study of Nonlinear Phenomena''.}\\
Dipartimento di Informatica\\
Universit\`a degli Studi di Verona\\
37134 Verona, Italy\\
[\bigskipamount]
\bf Yang Yang\thanks{This work was completed while the third-named author was visiting the Department of Mathematical Sciences at the Florida Institute of Technology, and she is grateful for the kind hospitality of the department. Project supported by NSFC-Tian Yuan Special Foundation (No. 11226116), Natural Science Foundation of Jiangsu Province of China for Young Scholars (No. BK2012109), and the China Scholarship Council (No. 201208320435).}\\
School of Science\\
Jiangnan University\\
Wuxi, 214122, China}
\date{}

\begin{document}

\maketitle

\begin{abstract}
We prove a bifurcation and multiplicity result that is independent of the dimension $N$ for a critical $p$-Laplacian problem that is the analog of the Brezis-Nirenberg problem for the quasilinear case. This extends a result in the literature for the semilinear case $p = 2$ to all $p \in (1,\infty)$. In particular, it gives a new existence result when $N < p^2$. When $p \ne 2$ the nonlinear operator $- \Delta_p$ has no linear eigenspaces, so our extension is nontrivial and requires a new abstract critical point theorem that is not based on linear subspaces. We prove a new abstract result based on a pseudo-index related to the $\Z_2$-cohomological index that is applicable here.
\end{abstract}

\newpage

\section{Introduction and main results}

Elliptic problems with critical nonlinearities have been widely studied in the literature. Let $\Omega$ be a bounded domain in $\R^N,\, N \ge 2$ with Lipschitz boundary. In the celebrated paper \cite{MR709644}, Br{\'e}zis and Nirenberg considered the problem
\begin{equation} \label{1.1}
\left\{\begin{aligned}
- \Delta u & = \lambda u + |u|^{2^\ast - 2}\, u && \text{in } \Omega\\[10pt]
u & = 0 && \text{on } \bdry{\Omega}
\end{aligned}\right.
\end{equation}
when $N \ge 3$, where $2^\ast = 2N/(N - 2)$ is the critical Sobolev exponent. Among other things, they proved that this problem has a positive solution when $N \ge 4$ and $0 < \lambda < \lambda_1$, where $\lambda_1 > 0$ is the first Dirichlet eigenvalue of $- \Delta$ in $\Omega$. Capozzi et al.\! \cite{MR831041} extended this result by proving the existence of a nontrivial solution for all $\lambda > 0$ when $N \ge 4$. The existence of infinitely many solutions for all $\lambda > 0$ was established by Fortunato and Jannelli \cite{MR890056} when $N \ge 4$ and $\Omega$ is a ball, and by Devillanova and Solimini \cite{MR1919704} when $N \ge 7$ and $\Omega$ is an arbitrary bounded domain (see also Schechter and Zou \cite{MR2646823}).

Garc{\'{\i}}a Azorero and Peral Alonso \cite{MR912211}, Egnell \cite{MR956567}, and Guedda and V{\'e}ron \cite{MR1009077} studied the corresponding problem for the $p$-Laplacian
\begin{equation} \label{1.2}
\left\{\begin{aligned}
- \Delta_p\, u & = \lambda\, |u|^{p-2}\, u + |u|^{p^\ast - 2}\, u && \text{in } \Omega\\[10pt]
u & = 0 && \text{on } \bdry{\Omega}
\end{aligned}\right.
\end{equation}
when $1 < p < N$, where $\Delta_p\, u = \divg \left(|\nabla u|^{p-2}\, \nabla u\right)$ is the $p$-Laplacian of $u$ and $p^\ast = Np/(N - p)$. They proved that this problem has a positive solution when $N \ge p^2$ and $0 < \lambda < \lambda_1$, where $\lambda_1 > 0$ is the first Dirichlet eigenvalue of $- \Delta_p$ in $\Omega$. Degiovanni and Lancelotti \cite{MR2514055} extended their result by proving the existence of a nontrivial solution when $N \ge p^2$ and $\lambda > \lambda_1$ is not an eigenvalue, and when $N^2/(N + 1) > p^2$ and $\lambda \ge \lambda_1$ (see also Arioli and Gazzola \cite{MR1741848}). The existence of infinitely many solutions for all $\lambda > 0$ was recently established by Cao et al.\! \cite{MR2885967} when $N > p^2 + p$ (see also Wu and Huang \cite{MR3072257}).

On the other hand, Cerami et al.\! \cite{MR779872} proved the following bifurcation and multiplicity result for problem \eqref{1.1} that is independent of $N$ and $\Omega$. Let $0 < \lambda_1 < \lambda_2 \le \lambda_3 \le \cdots \to + \infty$ be the Dirichlet eigenvalues of $- \Delta$ in $\Omega$, repeated according to multiplicity, let
\[
S = \inf_{u \in H^1_0(\Omega) \setminus \set{0}}\, \frac{\norm[2]{\nabla u}^2}{\norm[2^\ast]{u}^2}
\]
be the best constant for the Sobolev imbedding $H^1_0(\Omega) \hookrightarrow L^{2^\ast}(\Omega)$ when $N \ge 3$, and let $\vol{\cdot}$ denote the Lebesgue measure in $\R^N$. If $\lambda_k \le \lambda < \lambda_{k+1}$ and
\[
\lambda > \lambda_{k+1} - \frac{S}{\vol{\Omega}^{2/N}},
\]
and $m$ denotes the multiplicity of $\lambda_{k+1}$, then problem \eqref{1.1} has $m$ distinct pairs of nontrivial solutions $\pm\, u^\lambda_j,\, j = 1,\dots,m$ such that $u^\lambda_j \to 0$ as $\lambda \nearrow \lambda_{k+1}$ (see \cite[Theorem 1.1]{MR779872}).

In the present paper we extend the above bifurcation and multiplicity result to the $p$\nobreakdash-Laplacian problem \eqref{1.2}. This extension to the quasilinear case is quite nontrivial. Indeed, the linking argument based on eigenspaces of $- \Delta$ in \cite{MR779872} does not work when $p \ne 2$ since the nonlinear operator $- \Delta_p$ does not have linear eigenspaces. We will use a more general construction based on sublevel sets as in Perera and Szulkin \cite{MR2153141} (see also Perera et al.\! \cite[Proposition 3.23]{MR2640827}). Moreover, the standard sequence of eigenvalues of $- \Delta_p$ based on the genus does not give enough information about the structure of the sublevel sets to carry out this linking construction. Therefore we will use a different sequence of eigenvalues introduced in Perera \cite{MR1998432} that is based on a cohomological index.

The $\Z_2$-cohomological index of Fadell and Rabinowitz \cite{MR57:17677} is defined as follows. Let $W$ be a Banach space and let $\A$ denote the class of symmetric subsets of $W \setminus \set{0}$. For $A \in \A$, let $\overline{A} = A/\Z_2$ be the quotient space of $A$ with each $u$ and $-u$ identified, let $f : \overline{A} \to \RP^\infty$ be the classifying map of $\overline{A}$, and let $f^\ast : H^\ast(\RP^\infty) \to H^\ast(\overline{A})$ be the induced homomorphism of the Alexander-Spanier cohomology rings. The cohomological index of $A$ is defined by
\[
i(A) = \begin{cases}
\sup \set{m \ge 1 : f^\ast(\omega^{m-1}) \ne 0}, & A \ne \emptyset\\[5pt]
0, & A = \emptyset,
\end{cases}
\]
where $\omega \in H^1(\RP^\infty)$ is the generator of the polynomial ring $H^\ast(\RP^\infty) = \Z_2[\omega]$. For example, the classifying map of the unit sphere $S^{m-1}$ in $\R^m,\, m \ge 1$ is the inclusion $\RP^{m-1} \incl \RP^\infty$, which induces isomorphisms on $H^q$ for $q \le m - 1$, so $i(S^{m-1}) = m$.

The Dirichlet spectrum of $- \Delta_p$ in $\Omega$ consists of those $\lambda \in \R$ for which the problem
\begin{equation} \label{1.6}
\left\{\begin{aligned}
- \Delta_p\, u & = \lambda\, |u|^{p-2}\, u && \text{in } \Omega\\[10pt]
u & = 0 && \text{on } \bdry{\Omega}
\end{aligned}\right.
\end{equation}
has a nontrivial solution. Although a complete description of the spectrum is not yet known when $p \ne 2$, we can define an increasing and unbounded sequence of eigenvalues via a suitable minimax scheme. The standard scheme based on the genus does not give the index information necessary for our purposes here, so we will use the following scheme based on the cohomological index as in Perera \cite{MR1998432}. Let
\begin{equation} \label{1.9}
\Psi(u) = \frac{1}{\dint_\Omega |u|^p\, dx}, \quad u \in \M = \set{u \in W^{1,p}_0(\Omega) : \int_\Omega |\nabla u|^p\, dx = 1}.
\end{equation}
Then eigenvalues of problem \eqref{1.6} on $\M$ coincide with critical values of $\Psi$. We use the standard notation
\[
\Psi^a = \set{u \in \M : \Psi(u) \le a}, \quad \Psi_a = \set{u \in \M : \Psi(u) \ge a}, \quad a \in \R
\]
for the sublevel sets and superlevel sets, respectively. Let $\F$ denote the class of symmetric subsets of $\M$ and set
\[
\lambda_k := \inf_{M \in \F,\; i(M) \ge k}\, \sup_{u \in M}\, \Psi(u), \quad k \in \N.
\]
Then $0 < \lambda_1 < \lambda_2 \le \lambda_3 \le \cdots \to + \infty$ is a sequence of eigenvalues of problem \eqref{1.6} and
\begin{equation} \label{1.7}
\lambda_k < \lambda_{k+1} \implies i(\Psi^{\lambda_k}) = i(\M \setminus \Psi_{\lambda_{k+1}}) = k
\end{equation}
(see Perera et al.\! \cite[Propositions 3.52 and 3.53]{MR2640827}). Making essential use of \eqref{1.7}, we will prove the following theorem.

\begin{theorem} \label{Theorem 1.1}
Let
\begin{equation} \label{1.11}
S = \inf_{u \in W^{1,p}_0(\Omega) \setminus \set{0}}\, \frac{\norm[p]{\nabla u}^p}{\norm[p^\ast]{u}^p}.
\end{equation}
\begin{enumroman}
\item \label{Theorem 1.1.i} If
\[
\lambda_1 - \frac{S}{\vol{\Omega}^{p/N}} < \lambda < \lambda_1,
\]
then problem \eqref{1.2} has a pair of nontrivial solutions $\pm\, u^\lambda$ such that $u^\lambda \to 0$ as $\lambda \nearrow \lambda_1$.

\item \label{Theorem 1.1.ii} If $\lambda_k \le \lambda < \lambda_{k+1} = \cdots = \lambda_{k+m} < \lambda_{k+m+1}$ for some $k, m \in \N$ and
\begin{equation} \label{1.10}
\lambda > \lambda_{k+1} - \frac{S}{\vol{\Omega}^{p/N}},
\end{equation}
then problem \eqref{1.2} has $m$ distinct pairs of nontrivial solutions $\pm\, u^\lambda_j,\, j = 1,\dots,m$ such that $u^\lambda_j \to 0$ as $\lambda \nearrow \lambda_{k+1}$.
\end{enumroman}
\end{theorem}

\noindent
In particular, we have the following existence result, which is new when $N < p^2$.

\begin{corollary} \label{Corollary 1.2}
Problem \eqref{1.2} has a nontrivial solution for all $\lambda \in \ds{\bigcup_{k=1}^\infty} \big(\lambda_k - S/\vol{\Omega}^{p/N},\lambda_k\big)$.
\end{corollary}

\noindent
We note that $\lambda_1 > S/\vol{\Omega}^{p/N}$. Indeed, let $\varphi_1$ be an eigenfunction associated with $\lambda_1$. It is well-known that $S$ is not attained, so
\[
\lambda_1 = \frac{\norm[p]{\nabla \varphi_1}^p}{\norm[p]{\varphi_1}^p} > \frac{S \norm[p^\ast]{\varphi_1}^p}{\norm[p]{\varphi_1}^p} \ge \frac{S}{\vol{\Omega}^{p/N}}
\]
by the H\"older inequality.

\noindent
The abstract result of Bartolo et al. \cite{MR713209} used in \cite{MR779872} is based on linear subspaces and therefore cannot be used to prove our Theorem \ref{Theorem 1.1}. In the next section we will prove a more general critical point theorem based on a pseudo-index related to the cohomological index that is applicable here (see also Perera et al.\! \cite[Proposition 3.44]{MR2640827}).

\section{An abstract critical point theorem}

In this section we prove an abstract critical point theorem based on the cohomological index that we will use to prove Theorem \ref{Theorem 1.1}.
Let $W$ be a Banach space and let $\A$ denote the class of symmetric subsets of $W \setminus \set{0}$.
The following proposition summarizes the basic properties of the cohomological index. 

\begin{proposition}[Fadell-Rabinowitz \cite{MR57:17677}] \label{Proposition 2.1}
The index $i : \A \to \N \cup \set{0,\infty}$ has the following properties:
\begin{properties}{i}
\item Definiteness: $i(A) = 0$ if and only if $A = \emptyset$;
\item \label{i2} Monotonicity: If there is an odd continuous map from $A$ to $B$ (in particular, if $A \subset B$), then $i(A) \le i(B)$. Thus, equality holds when the map is an odd homeomorphism;
\item Dimension: $i(A) \le \dim W$;
\item Continuity: If $A$ is closed, then there is a closed neighborhood $N \in \A$ of $A$ such that $i(N) = i(A)$. When $A$ is compact, $N$ may be chosen to be a $\delta$-neighborhood $N_\delta(A) = \set{u \in W : \dist{u}{A} \le \delta}$;
\item Subadditivity: If $A$ and $B$ are closed, then $i(A \cup B) \le i(A) + i(B)$;
\item \label{i6} Stability: If $SA$ is the suspension of $A \ne \emptyset$, obtained as the quotient space of $A \times [-1,1]$ with $A \times \set{1}$ and $A \times \set{-1}$ collapsed to different points, then $i(SA) = i(A) + 1$;
\item \label{i7} Piercing property: If $A$, $A_0$ and $A_1$ are closed, and $\varphi : A \times [0,1] \to A_0 \cup A_1$ is a continuous map such that $\varphi(-u,t) = - \varphi(u,t)$ for all $(u,t) \in A \times [0,1]$, $\varphi(A \times [0,1])$ is closed, $\varphi(A \times \set{0}) \subset A_0$ and $\varphi(A \times \set{1}) \subset A_1$, then $i(\varphi(A \times [0,1]) \cap A_0 \cap A_1) \ge i(A)$;
\item Neighborhood of zero: If $U$ is a bounded closed symmetric neighborhood of $0$, then $i(\bdry{U}) = \dim W$.
\end{properties}
\end{proposition}

\noindent
Let $\Phi$ be an even $C^1$-functional defined on $W$, and recall that $\Phi$ satisfies the Palais-Smale compactness condition at the level $c \in \R$, or \PS{c} for short, if every sequence $\seq{u_j} \subset W$ such that $\Phi(u_j) \to c$ and $\Phi'(u_j) \to 0$ has a convergent subsequence. Let $\A^\ast$ denote the class of symmetric subsets of $W$, let $r > 0$, let $S_r = \set{u \in W : \norm{u} = r}$, let $0 < b \le + \infty$, and let $\Gamma$ denote the group of odd homeomorphisms of $W$ that are the identity outside $\Phi^{-1}(0,b)$. The pseudo-index of $M \in \A^\ast$ related to $i$, $S_r$ and $\Gamma$ is defined by
\[
i^\ast(M) = \min_{\gamma \in \Gamma}\, i(\gamma(M) \cap S_r)
\]
(see Benci \cite{MR84c:58014}). The following critical point theorem generalizes \cite[Theorem 2.4]{MR713209}.

\begin{theorem} \label{Theorem 2.1}
Let $A_0,\, B_0$ be symmetric subsets of $S_1$ such that $A_0$ is compact, $B_0$ is closed, and
\[
i(A_0) \ge k + m, \qquad i(S_1 \setminus B_0) \le k
\]
for some integers $k \ge 0$ and $m \ge 1$. Assume that there exists $R > r$ such that
\[
\sup \Phi(A) \le 0 < \inf \Phi(B), \qquad \sup \Phi(X) < b,
\]
where $A = \set{Ru : u \in A_0}$, $B = \set{ru : u \in B_0}$, and $X = \set{tu : u \in A,\, 0 \le t \le 1}$. For $j = k + 1,\dots,k + m$, let
\[
\A_j^\ast = \set{M \in \A^\ast : M \text{ is compact and } i^\ast(M) \ge j}
\]
and set
\[
c_j^\ast := \inf_{M \in \A_j^\ast}\, \max_{u \in M}\, \Phi(u).
\]
Then
\[
\inf \Phi(B) \le c_{k+1}^\ast \le \dotsb \le c_{k+m}^\ast \le \sup \Phi(X),
\]
in particular, $0 < c_j^\ast < b$. If, in addition, $\Phi$ satisfies the {\em \PS{c}} condition for all $c \in (0,b)$, then each $c_j^\ast$ is a critical value of $\Phi$ and there are $m$ distinct pairs of associated critical points.
\end{theorem}

\begin{proof}
If $M \in \A_{k+1}^\ast$,
\[
i(S_r \setminus B) = i(S_1 \setminus B_0) \le k < k + 1 \le i^\ast(M) \le i(M \cap S_r)
\]
since $\id{W} \in \Gamma$. Hence $M$ intersects $B$ by \ref{i2} of Proposition \ref{Proposition 2.1}. It follows that $c_{k+1}^\ast \ge \inf \Phi(B)$.
If $\gamma \in \Gamma$, consider the continuous map
\[
\varphi : A \times [0,1] \to W, \quad \varphi(u,t) = \gamma(tu).
\]
We have $\varphi(A \times [0,1]) = \gamma(X)$, which is compact. Since $\gamma$ is odd, $\varphi(-u,t) = - \varphi(u,t)$ for all $(u,t) \in A \times [0,1]$ and $\varphi(A \times \set{0}) = \set{\gamma(0)} = \set{0}$. Since $\Phi \le 0$ on $A$, $\restr{\gamma}{A} = \id{A}$ and hence $\varphi(A \times \set{1}) = A$. Applying \ref{i7} with $\widetilde{A}_0 = \set{u \in W : \norm{u} \le r}$ and $\widetilde{A}_1 = \set{u \in W : \norm{u} \ge r}$ gives
\[
i(\gamma(X) \cap S_r) = i(\varphi(A \times [0,1]) \cap \widetilde{A}_0 \cap \widetilde{A}_1) \ge i(A) = i(A_0) \ge k + m.
\]
It follows that $i^\ast(X) \ge k + m$. So $X \in \A_{k+m}^\ast$ and hence $c_{k+m}^\ast \le \sup \Phi(X)$.
The rest now follows from standard results in critical point theory, see e.g.\ \cite{MR2640827}.
\end{proof}

\begin{remark}
Constructions similar to the one in the proof of Theorem \ref{Theorem 2.1} have been used in Fadell and Rabinowitz \cite{MR57:17677} to prove bifurcation results for Hamiltonian systems, and in Perera and Szulkin \cite{MR2153141} to obtain nontrivial solutions of $p$-Laplacian problems with nonlinearities that interact with the spectrum. See also \cite[Proposition 3.44]{MR2640827}.
\end{remark}

\section{Proof of Theorem \ref{Theorem 1.1}}

In this section we prove Theorem \ref{Theorem 1.1}. We only give the proof of \ref{Theorem 1.1.ii}. Proof of \ref{Theorem 1.1.i} is similar and simpler. Solutions of problem \eqref{1.2} coincide with critical points of the $C^1$-functional
\[
\Phi(u) = \int_\Omega \left[\frac{1}{p}\, |\nabla u|^p - \frac{\lambda}{p}\, |u|^p - \frac{1}{p^\ast}\, |u|^{p^\ast}\right] dx, \quad u \in W^{1,p}_0(\Omega).
\]
By Guedda and V{\'e}ron \cite[Theorem 3.4]{MR1009077}, $\Phi$ satisfies the \PS{c} condition for all $c < S^{N/p}/N$, so we apply Theorem \ref{Theorem 2.1} with $b = S^{N/p}/N$. By Degiovanni and Lancelotti \cite[Theorem 2.3]{MR2514055}, the sublevel set $\Psi^{\lambda_{k+m}}$ has a compact symmetric subset $A_0$ with
\[
i(A_0) = k + m.
\]
We take $B_0 = \Psi_{\lambda_{k+1}}$, so that
\[
i(S_1 \setminus B_0) = k
\]
by \eqref{1.7}. Let $R > r > 0$ and let $A$, $B$ and $X$ be as in Theorem \ref{Theorem 2.1}.
For $u \in \Psi_{\lambda_{k+1}}$,
\[
\Phi(ru) \ge \frac{r^p}{p} \left(1 - \frac{\lambda}{\lambda_{k+1}}\right) - \frac{r^{p^\ast}}{p^\ast\, S^{p^\ast/p}}
\]
by \eqref{1.11}. Since $\lambda < \lambda_{k+1}$ and $p^\ast > p$, it follows that $\inf \Phi(B) > 0$ if $r$ is sufficiently small. For $u \in A_0 \subset \Psi^{\lambda_{k+1}}$,
\[
\Phi(Ru) \le \frac{R^p}{p} \left(1 - \frac{\lambda}{\lambda_{k+1}}\right) - \frac{R^{p^\ast}}{p^\ast \vol{\Omega}^{p^\ast/N} \lambda_{k+1}^{p^\ast/p}}
\]
by the H\"older inequality, so there exists $R > r$ such that $\Phi \le 0$ on $A$. For $u \in X$,
\begin{align*}
\Phi(u) & \le  \frac{\lambda_{k+1} - \lambda}{p} \int_\Omega |u|^p\, dx - \frac{1}{p^\ast \vol{\Omega}^{p^\ast/N}} \left(\int_\Omega |u|^p\, dx\right)^{p^\ast/p}\\[10pt]
& \le  \sup_{\rho \ge 0}\, \left[\frac{(\lambda_{k+1} - \lambda)\, \rho}{p} - \frac{\rho^{p^\ast/p}}{p^\ast \vol{\Omega}^{p^\ast/N}}\right]\\[10pt]
& =  \frac{\vol{\Omega}}{N}\, (\lambda_{k+1} - \lambda)^{N/p}.
\end{align*}
So
\[
\sup \Phi(X) \le \frac{\vol{\Omega}}{N}\, (\lambda_{k+1} - \lambda)^{N/p} < \frac{S^{N/p}}{N}
\]
by \eqref{1.10}.
Theorem \ref{Theorem 2.1} now gives $m$ distinct pairs of (nontrivial) critical points $\pm\, u^\lambda_j,\, j = 1,\dots,m$ of $\Phi$ such that
\begin{equation} \label{4.4}
0 < \Phi(u^\lambda_j) \le \frac{\vol{\Omega}}{N}\, (\lambda_{k+1} - \lambda)^{N/p} \to 0 \text{ as } \lambda \nearrow \lambda_{k+1}.
\end{equation}
Then
\[
\|u^\lambda_j\|_{p^\ast}^{p^\ast} = N \big(\Phi(u^\lambda_j) - \frac{1}{p}\, \Phi'(u^\lambda_j)\, u^\lambda_j\big) = N \Phi(u^\lambda_j) \to 0
\]
and hence $u^\lambda_j \to 0$ in $L^p(\Omega)$ also by the H\"older inequality, so
\[
\big\|\nabla u^\lambda_j\big\|^p_p = p \Phi(u^\lambda_j) + \lambda\, \|u^\lambda_j\|_p^p + \frac{p}{p^\ast}\, \|u^\lambda_j\|_{p^\ast}^{p^\ast} \to 0.
\]
This completes the proof of Theorem \ref{Theorem 1.1}.

\end{document}